\documentclass{amsart}
\usepackage{amsmath,amssymb,amsfonts,latexsym,cite}
\usepackage[foot]{amsaddr}
\theoremstyle{plain}
\newtheorem{theorem}{Theorem}
\numberwithin{theorem}{section}
\newtheorem{lemma}{Lemma}
\numberwithin{lemma}{section}
\numberwithin{equation}{section}
\bibstyle{plain}

\begin{document}
\title[Weighted estimates of the Cayley transform method...]{Weighted estimates of the Cayley transform method
for boundary value problems \\in a Banach space}
\author{V. L. Makarov}
\address{Institute of Mathematics of the National Academy of Sciences of Ukraine,
3 Tereshchenkivska Str., 01024 Kyiv,  Ukraine}
\email{makarovimath@gmail.com}
\author{N. V. Mayko}
\address{Taras Shevchenko National University of Kyiv\\
 64/13 Volodymyrska Str., 01601 Kyiv, Ukraine}
\email{mayko@knu.ua}

%\keywords{Boundary value problem (BVP), Banach space,
%Cayley transform, exponentially convergent algorithm,
%algorithm without saturation of accuracy, weighted estimate}

%\date{March 8, 2020}
\begin{abstract}
We consider the boundary value problems (BVPs) for linear
second-order  ODEs
with a strongly positive operator coefficient in a Banach space.
The solutions are given in the form of the infinite series
by means of the Cayley transform of the operator,
the Meixner type polynomials of the independent variable,
the operator Green function
and the Fourier series representation for the right-hand side of the equation.
The approximate solution of each problem is a partial sum
of $N$ (or expressed through $N$) summands.
We prove the weighted error estimates depending
on the discretization parameter $N$,
the distance of the independent variable
to the boundary points of the interval
and some smoothness properties of the input data.\vspace{0.5cm}

\vspace{0.5cm}

\end{abstract}

\maketitle

\section{Introduction}\label{sec1}
In numerical methods for BVPs the error estimate is quite often
expressed only through a particular discretization parameter
(e.\,g. a mesh step $h$, a number $N$ of summands
in a truncated sum of an infinite series, etc.).
However, some other features can also be important for the convergence rate.
In one such case, when we deal with the Dirichlet boundary condition
and therefore near the boundary an approximate solution is expected to be more precise,
it is natural to additionally estimate the error through the distance
from a point inside the domain to its boundary.
The impact of the Dirichlet boundary condition on the accuracy of an approximate solution
can be called \textit{a boundary effect} in the sense of\cite{Makarov1989}
and is usually evaluated by means of an appropriate weight function.
This aspect is investigated
in~\cite{Galba, MaykoRyabichev2016, Mayko2018-1, Mayko2018-2,
MakarovMayko2019-1,MakarovMayko2019-2, GavrilyukMakarovMayko2019-1}
(see also~\cite{KhoromskijMelenk}
where a method with  linear-logarithmic complexity
for solving elliptic problems with rough boundary data or geometry is presented).
In a similar way, it is possible to study the influence of the initial condition
and express the error estimate through both a discretization parameter and
the proximity of the time variable to the initial point
\cite{MakarovDemkiv2003, Gavrilyuk2010, Mayko2017}.
This aspect is of both theoretical and practical interest;
however, there are relatively few publications on the subject.

Another facet is connected with the dependence of the accuracy order
on differential properties of the input data
in the sense of~\cite{Babenko}.
Namely, a method is considered to be
\textit{a method without saturation of accuracy}
if an increase in smoothness of an exact solution
automatically boosts the convergence rate of an approximate solution.
It is also expected for a method to be \textit{exponentially convergent}
if input data are infinitely smooth in some sense.

Both of these concepts are  developed in  \cite{GavrilyukMakarovMayko2019-2},
where the proposed approximations of abstract differential equations
take into account the boundary effect and  do not have saturation of accuracy.
More specifically, in \cite{GavrilyukMakarovMayko2019-2} we consider
in a Hilbert space $H$
the  BVP for a linear second-order differential equation :
\begin{equation}\label{intr1}
\begin{split}
&\frac{d^{2} u(x)}{dx^{2} } -Au(x)=-f(x),\quad u(0)=0,\quad u(1)=0,
\end{split}
\end{equation}
where $A$ is a densely defined  self-adjoint positive definite operator
with  the spectral set $\Sigma (A)\subset [\lambda _{0} ,+\infty )$, $\lambda _{0} >0$.
To give the general idea of our approach,
we  briefly describe obtaining one of the two approximations for the solution $u(x)$.

By using the Fourier series representation of the right-hand side $f(x)$:
\begin{equation}\label{intr2}
\begin{split}
&f(x)=\sum _{k=1}^{\infty }\sqrt{2} \sin  (2k\pi x)f_{s,k}  +f_{0} +\sum _{k=1}^{\infty }\sqrt{2} \cos  (2k\pi x)f_{c,k},
\end{split}
\end{equation}
where
\begin{equation}\label{intr3}
\begin{split}
f_{s,k} &=\int _{0}^{1}f(x)\sqrt{2} \sin  (2k\pi x)dx ,\quad f_{c,k} =\int _{0}^{1}f(x)\sqrt{2} \cos  (2k\pi x)dx ,\quad k=1,2,\ldots ,\\
f_{0} &=\int _{0}^{1}f(x)dx,
\end{split}
\end{equation}
we present the solution $u(x)$ in the form
\begin{equation}\label{intr4}
\begin{split}
u(x)&=\sum _{k=1}^{\infty }\sqrt{2} \sin (2k\pi x)\left[(2k\pi )^{2} I + A\right]^{-1} f_{s,k}\\
&\quad +\big(A \,\sinh\sqrt{A}\, \big)^{-1} \left\{\sinh\sqrt{A} -\sinh\big(\sqrt{A} (1-x)\big)-\sinh\big(\sqrt{A} x\big)\right\}f_{0}\\
&\quad +\sum _{k=1}^{\infty }\sqrt{2} \left[(2k\pi )^{2} I+A\right]^{-1} \sinh^{-1} \sqrt{A} \\
&\quad \times \Big\{\cos (2k\pi x)\sinh\sqrt{A} -  \sinh\big(\sqrt{A} (1-x)\big)-\sinh\big(\sqrt{A} x\big)\Big\}f_{c,k}.
\end{split}
\end{equation}
Taking here the partial sum of $N$ summands in each infinite series,
we obtain for $u(x)$ the approximation $u_N(x)$
(for another approximation $u_{N,M}$ and its accuracy order see details in \cite{GavrilyukMakarovMayko2019-2}).
Next we prove the assertion.
\begin{theorem}[\cite{GavrilyukMakarovMayko2019-2}]
Let $\sigma >0$,
$f_0\in D(A^{\sigma})$
and let $f_{c,k}$, $f_{s,k}$ in (\ref{intr2}) satisfy the conditions
\begin{equation*}
\begin{split}
&\left\| f_{s} \right\| _{\sigma } \equiv \Big(\sum _{k=1}^{\infty }\, k^{2\sigma } \left\| f_{s,k} \right\| ^{2}  \Big)^{1/2} <\infty , \quad
 \left\| f_{c} \right\| _{\sigma } \equiv \Big(\sum _{k=1}^{\infty }\, k^{2\sigma } \left\| f_{c,k} \right\| ^{2}  \Big)^{1/2} <\infty.
\end{split}
\end{equation*}
Then the following weighted estimate holds true:
\begin{equation*}
\begin{split}
&\left\|  \frac{u(x)-u_{N} (x)}{\min  (x,1-x)} \, \right\|
\le \frac{C}{N^{\sigma +1/2} }
\big(\left\| f_{s} \right\| _{\sigma } +\left\| f_{c} \right\| _{\sigma } +\left\| A^{\sigma } f_{0} \right\| \big),
\end{split}
\end{equation*}
where $C>0$ is a constant independent of $N$.
\end{theorem}

The aim of the present paper is to extend this study
to a more general case of a Banach space
and obtain some new results.

The paper is organized as follows.
In Section~\ref{sec2} we consider the BVP for a linear second-order ODE
in a Banach space with a strongly positive operator coefficient,
i.e. a densely defined closed linear operator
under a certain assumption about its spectrum and resolvent.
We  introduce some useful notation for spaces and norms
and we prove a number of auxiliary inequalities used throughout the paper.
The exact solution of the BVP is represented in the form of the infinite series
involving the Meixner polynomials of the independent variable
and the Cayley transform of the operator coefficient.
That representation naturally gives rise to the approximate solution
in the form of the partial sum which is then studied
under various conditions about smoothness of the input data.
Namely, we prove two weighted error estimates
depending on both the discretization parameter $N$
(the number of summands in the partial sum)
and the distance to the boundary points of the interval.
In Section~\ref{sec5} we take a similar approach to the study
of the BVP for the inhomogeneous equation.
More specifically, we write down the exact solution
in the form of the infinite series
by the use of the operator Green function and the Fourier series representation
of the right-hand side of the equation.
The partial sum of this series produces the approximate solution
which is investigated under some smoothness assumptions concerning
the descending order of the Fourier coefficients.
Then we discuss the proven results
and make a few final comments.

\section{BVP for the homogeneous equation}\label{sec2}
We consider in a Banach space $E$ the boundary value problem
\begin{equation}\label{1}
\begin{split}
\frac{d^2u(x)}{dx^2} &-Au(x)=0,\quad x\in(0,1),\\
u(0)&=0,\quad u(1)=u_1,
\end{split}
\end{equation}
where $u(x)\!\!: [0,1]\to E$ is an unknown vector-valued function,
$u_1\in E$ is a given vector,
$A\!\!:E\to E$ \,is a densely defined closed linear operator
with the domain $D(A)$, the resolvent set $\rho(A)$
and the spectrum $\sigma(A)$.
Let $A$ satisfy the condition (see \cite[p.~69]{Pazy}):
there exit constants
$\varphi \in(0,\pi /2)$,
$\gamma >0$, $L >0$\, such that
\begin{equation}\label{2}
\Sigma \overset{\text{def}}{=}\{z\in\mathbb{C}\mid
\varphi\leqslant|\arg(z)| \leqslant\pi\}\cup
\{z\in\mathbb{C}\mid |z|\leqslant \gamma\} \subset \rho(A)
\end{equation}
and
\begin{equation}\label{3}
\|(zI-A)^{-1}\|\leqslant \frac{L}{1+|z|}\quad \forall z\in\Sigma.
\end{equation}
It is shown in~\cite{Pazy}  that strongly elliptic operators of order $2m$
in $L_p(\Omega)$ for $1\leqslant p \leqslant +\infty$
and a bounded domain $\Omega\subset \mathbb R^n$ with a smooth boundary $\partial \Omega$
produce important  examples  of such operators.

Note that in~\cite{GavrilyukMakarov1999}
a densely defined closed linear operator satisfying~(\ref{2}) and (\ref{3})
is called \textit{a strongly positive operator}.
It is proved~\cite[p.~708]{GavrilyukMakarov1999} that for \,$u_1\in D(A^\sigma)$,  $\sigma>1$,
the solution $u(x)$ of problem~(\ref{1}) can be represented in the form
\begin{equation}\label{4}
u(x)\equiv\sinh^{-1}(\sqrt{A})\sinh(x\sqrt{A})u_1
=\sum_{k=0}^{\infty}v_k(x)y_k;
\end{equation}
here
\begin{equation}\label{5}
y_k=(I+A)^{-1}Ay_{k-1}=[(I+A)^{-1}A]^k u_1,
\end{equation}
$(I+A)^{-1}A$ is the Cayley transform of the operator $A$,
and the functions $v_k(x)$ are defined
by the recurrent sequence of the integral equations
\begin{gather}\label{6}
v_k(x)=v_{k-1}(x)-\int_0^1 G_0(x,\xi) v_{k-1}(\xi)\,d\xi\quad x\in[0,\,1],\quad k=2,3,\ldots,\\
v_0(x)=x,\quad v_1(x)=-\frac{1}{3!}x(1-x^2),\notag
\end{gather}
where
\begin{equation*}
G_0(x,\xi)=
\begin{cases}
x(1-\xi)     &\text{if \,$x\leqslant\xi$,}\\
\xi(1-x)&\text{if \,$\xi\leqslant x$},
\end{cases}
\end{equation*}
is the Green function of the differential operator
\begin{equation*}
Lv(x)=-v''(x),\,x\in(0,\,1),\quad u(0)=0,\,u(1)=0.
\end{equation*}

Note that polynomials $v_k(x)$ are closely connected
with the Meixner polynomials~\cite{Meixner}
and have recently been studied in~\cite{MakarovDAN2019}.
The Meixner polynomials play the same role as the Laguerre polynomials
in the Cayley transform method for solving the abstract Cauchy problem
for a first-order differential equation
with a strongly positive operator coefficient \cite{GavrilyukMakarov2004}.

Next we recall some notation and terminology related
to certain classes of vectors from $E$ (see~\cite{GorbachukKnyazyuk1989}).
We denote by $C^{\infty}(A)$
the set of all infinitely differentiable vector of $A$:\,
$C^{\infty}(A)=\bigcap\limits_{n=0}^{\infty}D(A^n)$.
It is shown in~\cite{GorbachukKnyazyuk1989} that if a densely defined closed linear operator $A$
has at least one regular point,
then $C^{\infty}(A)$ is dense in $E$:\, $\overline{C^{\infty}(A)}=E$.

Let $(m_n)_{n=0}^{\infty}$ be a  nondecreasing sequence of positive numbers
and let $\nu>0$.
We denote by $C(A,(m_n),\nu)$ the Banach space of vectors $f$ from $C^{\infty}(A)$ with the norm
\begin{equation}\label{norm}
\|f\|_{C(A,\,(m_n),\,\nu)}= \sup_{n}\frac{\|A^n f\|}{\nu^n m_n}.
\end{equation}

The class $C(A,(m_n))\overset{\text{def}}{=}\bigcup\limits_{\nu>0}C(A, (m_n), \nu)$
for various $(m_n)$ is discussed in~\cite{GorbachukKnyazyuk1989}.
Namely, vectors from $C(A, (n^n))$ with $m_n=n^n$ are called analytic for the operator $A$~\cite{Nelson};
vectors from the Gevrey class of Roumieu type $C(A, (n^{n\beta}))$
with $m_n=n^{n\beta},\,\beta>1$,
are called ultradifferentiable~\cite{Beals},
and vectors from $C(A,(1))$ with $m_n\equiv 1$ are
known as vectors of exponential type~\cite{Radyno1985}.

%\section{Preliminary results}\label{sec3}
For convenience, in the next four lemmas we prove some auxiliary inequalities
which we will need throughout the paper. %later.
\begin{lemma}\label{lem0}
For a strongly positive operator $A$
the following estimate holds true:
\begin{equation}\label{6.1}
\left\|\left (I+\frac A j\right)^j A^{-j}\right\|
\leqslant  (L+1)^j \quad (j\in\mathbb{N}).
\end{equation}
\end{lemma}
\begin{proof}
From~(\ref{3}) we have  $\|A^{-1}\|\leqslant L$.
Therefore
\begin{equation*}
\begin{split}
\left\|\left (I+\frac A j\right)^j A^{-j}\right\|
&= \left\| \sum\limits_{s=0}^{j}\binom{j}{s} \left(\frac A j\right)^s A^{-j} \right \|
= \left\| \sum\limits_{s=0}^{j}\binom{j}{s} \frac{A^{-(j-s)}}{j^s} \right \|\\
&\leqslant \sum\limits_{s=0}^j \binom{j}{s}\frac{\|A^{-(j-s)}\|}{j^s}
\leqslant \sum\limits_{s=0}^j \binom{j}{s}\frac{\|A^{-1}\|^{j-s}}{j^s}\\
&\leqslant \sum\limits_{s=0}^j \binom{j}{s}\frac{L^{j-s}}{j^s}
=\left( L +\frac 1 j\right)^j \leqslant (L+1)^j.
\end{split}
\end{equation*}
\end{proof}
\begin{lemma}\label{lem1}
For $n>0$ and $\alpha >0$ the following inequality holds true:
\begin{equation*}
\max_{t\geqslant 1} \left[\Big(1-\frac{1}{t}\Big)^n t^{-\alpha}\right]
\leqslant \Big(\frac{\alpha}{e}\Big)^{\alpha} n^{-\alpha}.
\end{equation*}
\end{lemma}
\begin{proof}
We consider the function
%\begin{equation*}
$\varphi (t)=\left(1-\frac 1 t\right)^n t^{-\alpha}$,\, $t\geqslant 1$,
%\end{equation*}
and find the derivative
\begin{equation*}
\varphi'(t) = (t-1)^{n-1}t^{-\alpha-n-1}(n+\alpha-\alpha t).
\end{equation*}
Then we have
\begin{equation*}
\begin{split}
\max_{t\geqslant 1} \varphi(t)&=\varphi\Big( \frac{n+\alpha}{\alpha}\Big)\\
&=\Big(1-\frac{\alpha}{n+\alpha}\Big)^n \Big(\frac{n+\alpha}{\alpha}\Big)^{-\alpha}
=\Big(\frac{n}{n+\alpha}\Big)^{n+\alpha}\alpha^{\alpha}n^{-\alpha}
\leqslant e^{-\alpha}\alpha^{\alpha}n^{-\alpha}.
\end{split}
\end{equation*}
\end{proof}
\begin{lemma}\label{lem2}
For $n>\alpha >0$  the following inequality holds true:
\begin{equation*}
\max_{t>0}\left[\Big( \frac{t}{t+1}\Big)^n t^{-\alpha}\right]
\leqslant \alpha^{\alpha}n^{-\alpha}.
\end{equation*}
\end{lemma}
\begin{proof}
For the function
%\begin{equation*}
$\varphi(t)=\left(\frac{t}{t+1} \right)^n t^{-\alpha}$,\, $t>0$,
%\end{equation*}
we have the derivative
\begin{equation*}
\varphi'(t)=t^{n-\alpha-1}(t+1)^{-n-1}(n-\alpha-\alpha t),
\end{equation*}
which gives the estimate
\begin{equation*}
%\begin{split}
\max_{t>0}\varphi(t)=\varphi\Big(\frac{n-\alpha}{\alpha}\Big)
=\Big(\frac{n-\alpha}{n}\Big)^n\Big(\frac{n-\alpha}{\alpha}\Big)^{-\alpha}
=\Big(\frac{n-\alpha}{n}\Big)^{n-\alpha}\alpha^{\alpha}n^{-\alpha}
<\alpha^{\alpha}n^{-\alpha}.
%\end{split}
\end{equation*}
\end{proof}
\begin{lemma}\label{lem3}
The following estimate holds true:
\begin{equation*}
\max_{t\geqslant 0}
\left[\frac{t}{\left(\frac{t}{k}+1\right)(t+1)}\right]^k
=\left[\frac{\sqrt{k}}{\left(\frac{1}{\sqrt k}+1\right)(\sqrt{k}+1)}\right]^k
\leqslant ee^{-2\sqrt k} \quad  (k\in\mathbb N).
\end{equation*}
It is unimprovable in the asymptotic sense\footnote{We write
$f(x)\sim g(x)$ as $x\to x_0$ provided that $\lim\limits_{x\to x_0}\frac{f(x)}{g(x)}=1$.}:
\begin{equation*}
\left[\frac{\sqrt{k}}{\left(\frac{1}{\sqrt k}+1\right)(\sqrt{k}+1)}\right]^k
\sim ee^{-2\sqrt k} \quad \text{as} \quad k\to\infty.
\end{equation*}

\end{lemma}
\begin{proof}
Considering the function
%\begin{equation*}
$g(t)=\bigg[\frac{t}{\left(\frac{t}{k}+1\right)(t+1)}\bigg]^k$,\, $t\geqslant 0$,
%\end{equation*}
we get
\begin{equation*}
\frac{d \ln g(t)}{dt}=\frac{k(k-t^2)}{t(t+k)(t+1)},
\end{equation*}
which means that $g(t)$ takes its maximum at the point $t=\sqrt k$\,:
\begin{equation*}
\max_{t\geqslant 0} g(t) = g(\sqrt t)
= \left[\frac{\sqrt{k}}{\left(\frac{1}{\sqrt k}+1\right)(\sqrt{k}+1)}\right]^k
=\left( 1+\frac{1}{\sqrt{k}} \right)^{-2k}.
\end{equation*}
Making use of the inequality
\begin{equation*}
\bigg(1 +\frac 1{\sqrt k}    \bigg)^{-2k}
=\exp{\left\{-2k \ln\bigg( 1+\frac 1{\sqrt k}  \bigg)\right\}}
\leqslant \exp{\left\{-2k \bigg( \frac1{\sqrt k} - \frac 1{2k}\bigg)\right\}}
=ee^{-2\sqrt k}
\end{equation*}
and the relation\footnote{We write $f(x)=o(g(x))$ as $x\to x_0$
provided that $\lim\limits_{x\to x_0} \frac{f(x)}{g(x)}=0$.}
\begin{gather*}
%\begin{split}
\bigg(1 +\frac 1{\sqrt k}    \bigg)^{-2k}
=\exp{\left\{-2k \ln\bigg( 1+\frac 1{\sqrt k}  \bigg)\right\}}
=\exp\left\{-2k\bigg(
\frac{1}{\sqrt{k}}-\frac{1}{2k}+ o\Big(\frac{1}{k}\Big)
\bigg)\right\}\\
=\exp\left\{ 1-2\sqrt{k} +o(1)\right\}
\sim ee^{-2\sqrt{k}}\quad\text{as}\,\, k\to\infty,
%\end{split}
\end{gather*}
we arrive at the conclusion of the lemma.
\end{proof}

Now we can move on to the inequalities for $v_k(x)$.
\begin{lemma}\label{lem4}
For the polynomials $v_k(x)$,\,$k=1,2,\ldots$, the following estimates hold true:
\begin{equation}\label{7}
\left | \frac{v_k(x)}{\min(x,1-x)}\right |
\leqslant \frac{C_1}{k^{(1-\varepsilon_1)/2}},\quad x\in[0,\,1],
\end{equation}
where $0< \varepsilon_1 <1$,
$C_1=(\frac{1-\varepsilon_1}{e})^{(1-\varepsilon_1)/2}\frac 2 {\pi ^{1+\varepsilon_1}}
\zeta(1+\varepsilon_1)$, and
$\zeta(\cdot)$ is the Riemann zeta function;
\begin{equation}\label{8}
\left | \frac{v_k(x)}{\min(x,1-x)}\right | \leqslant \frac 1 3, \quad
x\in[0,\,1].
\end{equation}
\end{lemma}
\begin{proof}
We continue $v_k(x)$ oddly onto the interval $[-1,\,0]$
and then periodically onto the whole real axis.
First, we prove that $v_k(x)$ can be represented in the form
\begin{equation}\label{9}
v_k(x)=\sum_{p=1}^{\infty}
\sqrt 2 a_p^{(k)}\sin(p\pi x),\quad x\in [0,\,1],
\end{equation}
with
\begin{equation*}
a_p^{(k)}=\sqrt 2 \int\limits_0^1v_k(x)\sin (p\pi x)\, dx
=\frac{\sqrt 2(-1)^p}{(p\pi)^3}
\Big(1-\frac1{(p\pi)^2}\Big)^{k-1}.
\end{equation*}
We make use of the mathematical induction method.
For $k=1$ formula~(\ref{9}) is true since we have
\begin{equation*}
v_1(x)=\sum_{p=1}^{\infty}
\sqrt 2 a_p^{(1)}\sin (p\pi x), \quad x\in [0,\,1],
\end{equation*}
with
\begin{equation*}
a_p^{(1)}=\sqrt 2\int\limits_0^1 v_1(x)\sin (p\pi x)\, dx
= -\frac{\sqrt 2}{6}\int\limits_0^1 x(1-x^2)\sin (p\pi x)\,dx
=\frac{\sqrt 2 (-1)^p}{(p\pi)^3}.
\end{equation*}
Assuming that formula~(\ref{9}) is true for some $k\in\mathbb N$,
we will prove it for $k+1$. We have
\begin{equation*}
\begin{split}
v_{k+1}(x)&=v_k(x)-\int\limits_0^1 G_0(x,\xi)v_k(\xi)\,d\xi
=\sum_{p=1}^{\infty}\sqrt 2 a_p^{(k)}
\bigg(1-\frac 1{(p\pi)^2}\sin(p\pi x)
\bigg),
\end{split}
\end{equation*}
from where we get
\begin{equation*}
\begin{split}
a_p^{(k+1)}= a_p^{(k)}
\bigg( 1-\frac 1{(p\pi)^2} \bigg)
&=\frac{\sqrt 2 (-1)^p}{(p\pi)^3}\bigg(1-\frac 1{(p\pi)^2}\bigg)^k,
\end{split}
\end{equation*}
which gives formula~(\ref{9}).

Next we prove estimate~(\ref{7}). We have
\begin{equation*}
\begin{split}
\left |
\frac{v_k(x)}{\min(x,1-x)}\right |
&= \left |
\sum_{p=1}^{\infty}
\sqrt 2\frac{\sqrt 2 (-1)^p}{(p\pi)^3}
\bigg( 1-\frac 1{(p\pi)^2} \bigg)^{k-1}
\frac{\sin(p\pi x)}{\min(x,1-x)}
\right |\\
&\leqslant \sum _{p=1}^{\infty}
\frac 2{(p\pi)^{1+\varepsilon_1}}
\frac 1{\big((p\pi)^2 \big)^{(1-\varepsilon_1)/2}}
\bigg(1-\frac 1{(p\pi)^2} \bigg)^{k-1}.
\end{split}
\end{equation*}
Applying lemma~\ref{lem1} for $n=k-1$,\, $\alpha =(1-\varepsilon_1)/2$,\,
$k=2,3,\ldots$,\,$0 < \varepsilon_1 <1$, we have
\begin{equation*}
\begin{split}
\bigg | \frac{v_k (x)}{\min (x,1-x)}\bigg |
&\leqslant
\sup_{p\in \mathbb N}
\left[\frac 1 {\big((p\pi)^2\big)^{(1-\varepsilon_1)/2}}
\bigg(1-\frac 1{(p\pi)^2} \bigg)^{k-1}
\right]
\sum_{p=1}^{\infty} \frac 2{(p\pi)^{1+\varepsilon_1}}\\
&\leqslant
\sup_{t\geqslant 1}
\left [ \frac 1{t^{(1-\varepsilon_1)/2}} \bigg(1-\frac 1 t\bigg)^{k-1} \right]
\frac 2{{\pi}^{1+\varepsilon_1}} \zeta(1+\varepsilon_1)\\
&\leqslant
\bigg(\frac{1-\varepsilon_1}{2e}\bigg)^{(1-\varepsilon_1)/2}
(k-1)^{-(1-\varepsilon_1)/2}\frac 2{\pi^{1+\varepsilon_1}} \zeta(1+\varepsilon_1)\\
&\leqslant\bigg(\frac{1-\varepsilon_1}{e}\bigg)^{(1-\varepsilon_1)/2}
\frac 2{\pi^{1+\varepsilon_1}} \zeta(1+\varepsilon_1)
k^{-(1-\varepsilon_1)/2}.
\end{split}
\end{equation*}
For $k=1$ inequality~(\ref{7}) takes the form
\begin{equation*}
\bigg|\frac{v_1(x)}{\min(x,1-x)}
\bigg|\leqslant C_1,\quad x\in[0,\,1],
\end{equation*}
from where it follows that it is true since
\begin{equation*}
\max_{0\leqslant x \leqslant 1}
\left |\frac{v_1(x)}{\min(x, 1-x)}\right |
=\max_{0\leqslant x \leqslant 1}
\left |\frac{ -\frac 1{3!}x(1-x^2)} {\min(x, 1-x)}\right | = \frac 1 3
\quad \text{and}\quad C_1 > \frac 1 3 \quad \text{for}\,\, 0<\varepsilon_1<1.
\end{equation*}
Thus formula~(\ref{7}) is fulfilled for all $k\in\mathbb N$.

Now we prove estimate~(\ref{8}):
\begin{equation*}
\begin{split}
\left |
\frac{v_k(x)}{\min(x,1-x)}\right |
&= \left |
\sum_{p=1}^{\infty}
\sqrt 2\frac{\sqrt 2(-1)^p}{(p\pi)^3}
\bigg( 1-\frac 1{(p\pi)^2}\bigg)^{k-1}
\frac{\sin(p\pi x)}{\min(x,1-x)}
\right|\\
&\leqslant \sum_{p=1}^{\infty}\frac 2{(p\pi)^3}
\bigg(1-\frac 1{(p\pi)^2} \bigg)^{k-1}\frac{\vert\sin(p\pi x) \vert}{\min(x,1-x)}
\leqslant \sum_{p=1}^{\infty}\frac 2{(p\pi)^2} =\frac 1 3,\quad k=1,2,\ldots.
\end{split}
\end{equation*}
This completes the proof.
\end{proof}

In the following two lemmas we estimate  $\|y_k\|$.
\begin{lemma}\label{lem5}
Let $\sigma>0$,\,\,$0<\varepsilon_2<\min(1,\sigma)$,\,
$k>\sigma-\varepsilon_2$,\, $u_1\in D(A^{\sigma})$.
Then the following inequality holds true:
\begin{equation}\label{10}
\|y_k\| \leqslant\frac{C_2}{k^{(\sigma-\varepsilon_2)/2}}
\|A^{\sigma}u_1\|
\end{equation}
with $C_2=\frac{L}{\sin(\pi\varepsilon_2)}(\sigma-\varepsilon_2)^{(\sigma-\varepsilon_2)/2}$.
\end{lemma}
\begin{proof}
For estimating $\|y_k\|$, we use~({\ref{5}}) and apply integration along the path
$\Gamma=\Gamma_{+}\cup \Gamma_{-}$
consisting of two rays on the complex plane:
\begin{equation}\label{11}
\Gamma_{\pm}=\{ z\in \mathbb C\mid z=\rho e^{\pm i\varphi },\,\rho\in[0,\,+\infty)\}.
\end{equation}
Then we have
\begin{equation*}\label{12}
\begin{split}
\|y_k\|
&= \left \| [(I+A)^{-1}A]^k u_1\right\|
=\left \|
\frac 1{2\pi i}\int\limits_{\Gamma}\Big( \frac{z}{1+z}\Big)^k z^{-\sigma}
(zI-A)^{-1} A^{\sigma}u_1 dz \right \|\\
&\leqslant \frac 1{2\pi}
\int\limits_{\Gamma}\left| \frac z{1+z}\right|^k  \vert z\vert^{-\sigma}
\frac L {1+|z|} \vert dz\vert\, \|A^{\sigma} u_1\|.
\end{split}
\end{equation*}

Making use of the relations
\begin{gather*}
|z|= |\rho e^{\pm i \varphi}| =\rho,\quad
|dz| = |d\, (\rho e^{\pm i\varphi})| = |e^{\pm i \varphi} d\rho| = d\rho,\\
\left| \frac z {1+z}\right |^2
= \left | \frac{\rho e^{\pm i\varphi}}{1+\rho e^{\pm i\varphi}} \right |^2
=\frac{\rho^2}{1+2\rho \cos \varphi +\rho^2}
\leqslant \frac{\rho^2}{1+\rho^2},
\end{gather*}
we get
\begin{equation*}
\begin{split}
\|y_k\|
&\leq \frac L {\pi}
\int\limits_0^{+\infty}
\left( \frac{\rho^2}{1+\rho^2} \right)^{k/2}
\frac{\rho^{-\sigma}}{1+\rho}d\rho \,\|A^{\sigma}u_1\|\\
&\leqslant\frac L{\pi}\int\limits_0^{+\infty}
\left( \frac{\rho^2}{1+\rho^2}\right)^{k/2}
(\rho^2)^{-(\sigma-\varepsilon_2)/2}
\frac{d\rho}{\rho^{\varepsilon_2}(1+\rho)} \,\|A^{\sigma}u_1\|\\
&\leqslant \frac L{\pi}
\sup_{t>0}\left[ \left(\frac t{1+t} \right)^{k/2}t^{-(\sigma-\varepsilon_2)/2}\right]
\frac{\pi}{\sin(\pi\varepsilon_2)}\|A^{\sigma}u_1\|.
\end{split}
\end{equation*}
Applying lemma~\ref{lem2} with
$n=\frac k 2$ and $\alpha =\frac{\sigma-\varepsilon_2}{2}$, we obtain the inequality
\begin{equation*}
\|y_k\|
\leqslant \frac{L}{\pi}
\left(\frac{\sigma-\varepsilon_2}{k}\right)^{(\sigma-\varepsilon_2)/2}
\frac{\pi}{\sin{(\pi\varepsilon_2)}}
\|A^{\sigma}u_1\|,
\end{equation*}
and therefore (\ref{10}) is proven.
\end{proof}

\begin{lemma}\label{lem6}
Let $u_1 \in E$ satisfy the condition $u_1\in D(A^k)\,\, \forall k \in \mathbb N$.
Then the following estimate holds true:
\begin{equation}\label{13}
\|y_k\| \leqslant \frac{L e}{\sqrt{2}}\,e^{-2\sqrt{k}}\,
\frac{(\sqrt{k}+1)^2}{\sqrt{k}}    \|u_1\|_{C(A, \,(1), \,\nu)},
\quad \nu = \frac {\cos\varphi}{L+1}.
\end{equation}
\end{lemma}
\begin{proof}
Taking  representation~(\ref{5}) into consideration and integrating along the path
$\Gamma=\Gamma_+ \bigcup \Gamma_-$ (see~(\ref{11})), we have
\begin{equation*}
\begin{split}
\|y_k\| &=\left \| [(I+A)^{-1}A]^k u_1  \right\|\\
&= \Bigg\|
\frac{1}{2\pi i}\int\limits_{\Gamma}
\left ( \frac{z}{1+z}\right)^k\left ( 1+\frac z k\right)^{-k}
(zI-A)^{-1}\left( I+\frac A k\right)^k u_1 dz\Bigg\|\\
&\leqslant \frac 1{2\pi} \int\limits_{\Gamma}
\left| \frac z{1+z}\right|^k \left | 1+\frac z k\right|^{-k} \frac L{1+|z|} |dz|\,
\bigg\| \left( I +\frac A k\right)^k u_1\bigg\|\\
&= \frac L\pi \int\limits_0^{+\infty}
\left[\frac{\rho\cos\varphi}
{\sqrt{1+2\rho\cos\varphi+\rho^2}
\sqrt{1+2\frac \rho k\cos\varphi+\frac{\rho^2}{k^2}}}
\right]^{k-1}\\
&\times
\frac{\rho\cos\varphi}
{\sqrt{1+2\rho\cos\varphi+\rho^2}
\sqrt{1+2\frac \rho k\cos\varphi+\frac{\rho^2}{k^2}}}
\frac{d\rho}{1+\rho}\cos^{-k}\varphi
\bigg\| \left(I+\frac A k\right)^k A^{-k} A^{k}u_1\bigg\|
\end{split}
\end{equation*}
Next we make use of the relations
\begin{gather*}
\sqrt{1+2\rho \cos \varphi +\rho^2}
\geqslant 1+\rho\cos\varphi, \quad
\sqrt{1+2\frac {\rho}{k}\cos\varphi+\frac{\rho^2}{k^2}}
\geq 1+\frac{\rho}{k}\cos\varphi,\\
\sqrt{1+2\frac {\rho}{k}\cos\varphi+\frac{\rho^2}{k^2}}
\geqslant \sqrt{1+\frac{\rho^2}{k^2}}
\geqslant \sqrt{\frac{2\rho}{k}},\quad
|dz| = |d\, (\rho e^{\pm i\varphi})| = |e^{\pm i \varphi} d\rho| = d\rho
\end{gather*}
and lemma~\ref{lem0}. %estimate~(\ref{6.1}).
Then
\begin{equation*}
\begin{split}
\|y_k\|
&\leqslant \frac{L\sqrt{k}}{\pi \sqrt{2}}\int\limits_0^{+\infty}
\left[\frac{\rho\cos\varphi}{\left(1+\rho\cos\varphi\right)
\left(1+\frac{\rho\cos\varphi}{k}\right)} \right]^{k-1}\!\!
\frac{d\rho}{\sqrt{\rho}(1+\rho)} \cos^{-k}\varphi
\bigg\|\left(I+\frac A k\right)^k A^{-k}\bigg\| \|A^ku_1\|\\
&\leqslant \frac{L\sqrt{k}}{\pi \sqrt{2}}
\max_{t\geqslant 0}\left[ \frac{t}{(1+t)\big(1+\frac{t}{k}\big)}\right]^{k-1}
\int\limits_0^{+\infty}\frac{d\rho}{\sqrt{\rho}(1+\rho)}
\cos^{-k}\varphi(L+1)^k \|A^k u_1\|\\
&=\frac{L\sqrt{k}}{\sqrt{2}}
\left[ \frac{\sqrt{k}}{(1+\sqrt{k})\big(1+\frac{1}{\sqrt k}\big)}\right]^{k}
\left(\frac{\sqrt{k}+1}{\sqrt{k}}\right)^{2} \frac{\|A^k u_1\|}{\nu^k}.
\end{split}
\end{equation*}
Applying here lemma~\ref{lem3} and using norm~(\ref{norm}),
we obtain estimate~(\ref{13}).
\end{proof}

%\section{The approximate solution of the homogeneous equation}\label{sec4}
We approximate the exact solution $u(x)$ of BVP~(\ref{1})
by the partial sum of series~(\ref{4}):
\begin{equation}\label{14}
\begin{split}
u_N (x)=\sum\limits_{k=0}^{N} v_k(x)y_k.
\end{split}
\end{equation}
In the next two theorems, we study the error $u(x)-u_N(x)$
under various assumptions about smoothness of the vector $u_1$.
\begin{theorem}\label{th1}
Let  $u_1\in D(A^{\sigma})$, \,$\sigma > 1$.
Then the accuracy of the approximate solution~(\ref{14})
is characterized by the weighted estimate
\begin{equation}\label{15}
\begin{split}
\left\| \frac{u(x)-u_N(x)}{\min(x,1-x)}\right\|
\leqslant \frac{C}{N^{(\sigma-1-\varepsilon)/2}}
\|A^{\sigma}u_1\|,\quad x\in[0,\,1] \quad (N\geqslant \sigma-1),
\end{split}
\end{equation}
where $\varepsilon >0$ is an arbitrary small  number
and \,$C>0$ is a constant independent of  $N$. % and $u(x)$.
\end{theorem}
\begin{proof}
For
$0<\varepsilon_1<1$,
$0<\varepsilon_2<1$,
$1+\varepsilon_1+\varepsilon_2<\sigma$ and $N+1>\sigma-\varepsilon_2$,
the assumptions of lemma~\ref{lem4} and lemma~\ref{lem5} are fulfilled.
Then we get
\begin{equation*}
\begin{split}
\left \| \frac{u(x)-u_N(x)}{\min(x,1-x)} \right\|
&=\left\| \sum\limits_{k=N+1}^{\infty}
\frac{v_k(x)}{\min(x,1-x)}y_k\right\|
\leqslant \sum\limits_{k=N+1}^{\infty}
\left|\frac{v_k(x)}{\min(x,1-x)} \right|\|y_k\|\\
&\leqslant \sum\limits_{k=N+1}^{\infty}
\frac{C_1}{k^{(1-\varepsilon_1)/2}}
\frac{C_2}{k^{(\sigma-\varepsilon_2)/2}} \|A^{\sigma}u_1\|
\leqslant C_1C_2\int\limits_{N}^{+\infty}
\frac {dx}{x^{(1+\sigma-\varepsilon_1-\varepsilon_2)/2}}
\|A^{\sigma}u_1\|\\
&=\frac{2C_1C_2}{\sigma-1-\varepsilon_1-\varepsilon_2}
\frac{1}{N^{(\sigma-1-\varepsilon_1-\varepsilon_2)/2}}
\|A^{\sigma}u_1\|,
\end{split}
\end{equation*}
from where estimate~(\ref{15}) easily follows.
\end{proof}
\begin{theorem}\label{th2}
Let $u_1\in C(A,\,(1),\,\nu)$ with $\nu=\frac{\cos\varphi}{L+1}$.
Then the accuracy of the approximate solution~(\ref{14})
is characterized by the weighted estimate
\begin{equation}\label{16}
\left\| \frac{u(x)-u_N(x)}{\min(x,1-x)}\right\|\leqslant
\frac{C e^{-\sqrt{N+1}}}{(N+1)^{1/2-\varepsilon}}\|u_1\|_{C(A,\,(1),\,\nu)},
\quad x\in[0,\,1] \quad (N\in\mathbb N),
\end{equation}
where $\varepsilon >0$ is an arbitrary small  number and
$C>0$ is a constant independent of  $N$. % and $u(x)$.
\end{theorem}
\begin{proof}
Applying lemma~\ref{lem4} and lemma~\ref{lem6}, we have
\begin{equation}\label{17}
\begin{split}
\left \| \frac{u(x)-u_N(x)}{\min(x,1-x)} \right\|
&=\left\| \sum\limits_{k=N+1}^{\infty}
\frac{v_k(x)}{\min(x,1-x)}y_k\right\|
\leqslant \sum\limits_{k=N+1}^{\infty}
\left|\frac{v_k(x)}{\min(x,1-x)} \right|\|y_k\|\\
&\leqslant \sum\limits_{k=N+1}^{\infty}
\frac{C_1}{k^{(1-\varepsilon_1)/2}}
\frac{L e}{\sqrt{2}}\,e^{-2\sqrt{k}}\,
\frac{(\sqrt{k}+1)^2}{\sqrt{k}}\|u_1\|_{C(A, \,(1), \,\nu)}\\
&\leqslant \frac{C_1 L e}{\sqrt{2}}
\frac{e^{-\sqrt{N+1}}}{(N+1)^{(1-\varepsilon_1)/2}}
\sum\limits_{k=1}^{\infty}
e^{-\sqrt{k}}\, \frac{(\sqrt{k}+1)^2}{\sqrt{k}}
\|u_1\|_{C(A,\,(1),\,\nu)}.
\end{split}
\end{equation}
Denoting by $S$ the sum of the convergent number series:
$S=\sum\limits_{k=1}^{\infty}
e^{-\sqrt{k}}\, \frac{(\sqrt{k}+1)^2}{\sqrt{k}}=8{.}152349342\ldots$,
and putting $C=\frac{C_1 L Se}{\sqrt{2}}$,
we arrive at inequality~(\ref{16}).
\end{proof}

\section{BVP for the inhomogeneous equation}\label{sec5}
In this section, we consider the BVP
\begin{equation}\label{18}
\begin{split}
\frac{d^2u(x)}{dx^2} &-Au(x)=-f(x),\quad x\in(0,1),\\
u(0)&=0,\quad u(1)=0,
\end{split}
\end{equation}
with the operator $A$ satisfying the same conditions
as in Section~\ref{sec2}.

To write down the solution $u(x)$ in a convenient way,
we make use of the Fourier series representation
for the right-hand side $f(x)$:
\begin{equation}\label{19}
f(x)=\sum_{k=1}^{\infty}\sqrt{2}\sin(2k\pi x)f_{s,k}+f_0+
\sum\limits_{k=1}^{\infty}\cos(2k\pi x)f_{c,k}
\end{equation}
with
\begin{equation}\label{20}
\begin{split}
f_0&=\int\limits_0^1 f(x)dx,\\
f_{s,k}&=\int\limits_0^1 f(x)\sqrt{2}\sin(2k\pi x)dx,\quad
f_{c,k}=\int\limits_0^1 f(x)\sqrt{2}\cos(2k\pi x)dx, \,\, k=1,2,\ldots.
\end{split}
\end{equation}
Applying the operator Green function
\begin{equation*}
G(x,\xi; A)=(\sqrt{A}\sinh\sqrt{A})^{-1}
\begin{cases}
\sinh(\sqrt{A}x)\sinh\big(\sqrt{A}(1-\xi)\big)     &\text{if \,$x\leqslant\xi$,}\\
\sinh(\sqrt{A}\xi)\sinh\big(\sqrt{A}(1-x)\big)     &\text{if \,$\xi\leqslant x$},
\end{cases}
\end{equation*}
we modify $u(x)$ as follows:
\begin{equation*}
\begin{split}
u(x)&= \int\limits_0^1 G(x,\xi;A)f(\xi)\,d\xi\\
&=\big(\sqrt{A}\sinh\sqrt{A}\,\big)^{-1}\sinh\big(\sqrt{A}(1-x)\big)
\int\limits_0^x \sinh(\sqrt{A}\xi)f(\xi)d\xi\\
&\quad+\big(\sqrt{A}\sinh\sqrt{A}\,\big)^{-1}\sinh(\sqrt{A}x)
\int\limits_x^1 \sinh\big(\sqrt{A}(1-\xi)\big)f(\xi)d\xi
\end{split}
\end{equation*}
\begin{equation}\label{21}
\begin{split}
&=\sum\limits_{k=1}^{\infty}
\sqrt{2}\sin(2k\pi x)\big[(2k\pi)^2I+A\big]^{-1}f_{s,k}\\
&\quad+ (A\sinh\sqrt{A})^{-1}
\Big[\sinh\sqrt{A} -\sinh\big(\sqrt{A}(1-x)\big)-\sinh(\sqrt{A}x)\Big]f_0\\
&\quad\sum\limits_{k=1}^{\infty}\sqrt{2}\big[(2k\pi)^2I+A\big]^{-1}\sinh^{-1}\sqrt{A}\\
&\quad \times
\Big[\cos(2k\pi x)\sinh\sqrt{A}-\sinh\big(\sqrt{A}(1-x)\big)-\sinh(\sqrt{A}x)  \Big]f_{c,k}.
\end{split}
\end{equation}
Bearing~(\ref{4}) in mind, we then rewrite~(\ref{21}) in the form
\begin{equation}\label{22}
\begin{split}
u(x)&=\sum_{k=1}^{\infty}\sqrt{2}\sin(2k\pi x)\big[ (2k\pi)^2I+A \big]^{-1}f_{s,k}\\
&\quad+A^{-1}
\bigg\{I-\sum_{j=0}^{\infty} \big[v_j(1-x)+v_j(x)\big]
\big[(I+A)^{-1}A\big]^j\bigg\}f_0
\end{split}
\end{equation}
\begin{equation*}
\begin{split}
&+\sum_{k=1}^{\infty}\sqrt{2}\big[(2k\pi)^2 I+A\big]^{-1}
\bigg\{\cos(2k\pi x)I
-\sum_{j=0}^{\infty}
\big[v_j(1-x)+v_j(x)\big]\big[(I+A)^{-1}A\big]^{j}\bigg\}f_{c,k}.
\end{split}
\end{equation*}
This representation yields the following approximation:
\begin{equation}\label{23}
\begin{split}
u_{N,M}(x)=&\sum_{k=1}^{N}\sqrt{2}\sin(2k\pi x)\big[ (2k\pi)^2I+A \big]^{-1}f_{s,k}\\
&\quad+A^{-1}\bigg\{I-\sum_{j=0}^{M} \big[v_j(1-x)+v_j(x)\big]
\big[(I+A)^{-1}A\big]^j\bigg\}f_0
\end{split}
\end{equation}
\begin{equation*}
\begin{split}
+\sum_{k=1}^{N}\sqrt{2}\big[(2k\pi)^2 I+A\big]^{-1}
\bigg\{\cos(2k\pi x)I
-\sum_{j=0}^{M}
\big[v_j(1-x)+v_j(x)\big]\big[(I+A)^{-1}A\big]^{j}\bigg\}f_{c,k}.
\end{split}
\end{equation*}
Next we study the accuracy of $u_{N,M}(x)$.
To this end, we introduce the error $u(x)-u_{N, M}(x)$
as the sum:
\begin{equation}\label{24}
u(x)-u_{N,M}(x)=\sum_{k=1}^{5}D_k,
\end{equation}
where
\begin{equation}\label{25}
\begin{split}
D_1&=\sum_{k=N+1}^{\infty}
\sqrt{2}\sin(2k\pi x) \big[(2k\pi)^2I+A\big]^{-1}f_{s,k},\\
D_2&=-\sum_{j=M+1}^{\infty}\big[v_j(1-x)+v_j(x)\big]A^{-1}\big[(I+A)^{-1}A\big]^j f_0,\\
D_3&=\sum_{k=N+1}^{\infty}\sqrt{2}\big[\cos(2k\pi x)-1\big]\big[(2k\pi )^2 I+A\big]^{-1}f_{c,k},\\
D_4&=-\sum_{k=1}^{\infty}\sqrt{2}\big[(2k\pi)^2I+A\big]^{-1}
\sum_{j=M+1}^{\infty}\big[v_j(1-x)+v_j(x)\big]\big[(I+A)^{-1}A\big]^jf_{c,k},\\
D_{5}&=-\sum_{k=N+1}^{\infty}\sqrt{2}\big[(2k\pi)^2I+A\big]^{-1}
\sum_{j=1}^{M}\big[v_j(1-x)+v_j(x)\big]\big[(I+A)^{-1}A\big]^jf_{c,k}.
\end{split}
\end{equation}

%\section{The approximate solution of the inhomogeneous equation}\label{sec6}
Now we prove some error estimates for approximation~(\ref{23})
under certain assumptions about smoothness of $f(x)$
in term of smoothness of the coefficients $f_0$, $f_{c,k}$, $f_{s,k}$ in~(\ref{19}).
\begin{theorem}\label{th3}
Let  the following conditions be fulfilled:
\begin{gather*}
\sigma >0,\quad
f_0\in D(A^{\sigma}),\quad f_{c,k}\in D(A^{\sigma})\quad\forall k\in\mathbb N,\\
\|f_s\|_{\sigma} \overset{\text{def}}{=}
\left\{\sum_{k=1}^{\infty}k^{\sigma +1}\|f_{s,k}\|^2\right\}^{1/2}<\infty,\quad\quad
\|f_c\|_{\sigma} \overset{\text{def}}{=}
\left\{\sum_{k=1}^{\infty}k^{\sigma +1}\|f_{c,k}\|^2\right\}^{1/2}<\infty,\\
%\|A^{\sigma}f_0\|< \infty,\quad
\|f_c\|_{A^{\sigma}}\overset{\text{def}}{=}  \sum_{k=1}^{\infty}\|A^{\sigma}f_{c,k}\|<\infty.
\end{gather*}
Then the accuracy of the approximate solution~(\ref{23})
is characterized by the weighted estimate
\begin{equation}\label{26}
\begin{split}
\left\| \frac{u(x)-u_{N,N}(x)}{\min(x,1-x)}\right\|
&\leqslant \frac{C}{N^{(\sigma-\varepsilon)/2}}
\Big(\|f_s\|_{\sigma}+\|f_c\|_{\sigma}+\|A^{\sigma}f_0\|
+\|f_c\|_{A^{\sigma}}\Big),\\
& x\in[0,\,1] \quad (N\geqslant \sigma),
\end{split}
\end{equation}
where $\varepsilon >0$ is an arbitrary small  number and
$C>0$ is a constant independent of  $N$. % and $u(x)$.
\end{theorem}
\begin{proof}
We focus now on estimating each summand in~(\ref{25}).
Throughout the whole proof, we use the integration path~(\ref{11}).
We will also need the relations
\begin{equation}\label{rel}
\begin{split}
\big|(2k\pi)^2+z\big| &=
\big|(2k\pi)^2+\rho e^{\pm i\varphi}\big|
=\sqrt{(2k\pi)^4+2(2k\pi)^2\rho\cos\varphi+\rho^2}\\
&\geqslant
\sqrt{(2k\pi)^4+\rho^2} \geqslant \sqrt{2(2k\pi)^2\rho}
\geqslant 2k\pi\sqrt{2\rho},\\
\big|(2k\pi)^2+z\big| &=
\big|(2k\pi)^2+\rho e^{\pm i\varphi}\big|
=\sqrt{(2k\pi)^4+2(2k\pi)^2\rho\cos\varphi+\rho^2}\geqslant \rho,\\
\left | \frac{z}{1+z}\right |^2
&=\left| \frac{\rho e^{\pm i\varphi}}{1+\rho e^{i\varphi}}\right|^2
= \frac{\rho^2}{1+2\rho\cos\varphi +\rho^2}\leqslant \frac{\rho^2}{1+\rho^2}
\end{split}
\end{equation}
and the Cauchy\textendash Bunyakovsky\textendash Schwarz inequality for number series.

Then for $D_1$ we get
\begin{equation}\label{27}
\begin{split}
\left \| \frac{D_1}{\min(x,1-x)} \right\| &=
\left\|  \sum_{k=N+1}^{\infty}
\frac{\sqrt{2}\sin(2k\pi x)}{\min(x,1-x)}
\big[(2k\pi)^2I+A\big]^{-1}f_{s,k} \right\|\\
&\leqslant \sum_{k=N+1}^{\infty}
\frac{\sqrt{2}|\sin(2k\pi x)|}{\min(x,1-x)}
\left\| \frac{1}{2\pi i}\int\limits_{\Gamma}\frac{1}{(2k\pi)^2+z}(zI-A)^{-1}f_{s,k}dz \right\|\\
&\leqslant
\sum_{k=N+1}^{\infty}
\frac{\sqrt{2}2k\pi }{2\pi}
\int\limits_{\Gamma}\frac{1}{|(2k\pi)^2+z|}
\frac{L}{1+|z|}|dz|\,\|f_{s,k}\|\\
&\leqslant
\frac {L}{\pi}\sum_{k=N+1}^{\infty}
\int\limits_0^{+\infty}\frac{d\rho}{\sqrt{\rho}(1+\rho)}\|f_{s,k}\|
=L\sum_{k=N+1}^{\infty}
\frac{1}{k^{(\sigma+1)/2}} k^{(\sigma+1)/2}\|f_{s,k}\|\\
&\leqslant L \left\{\sum_{k=N+1}^{\infty}\frac 1{k^{\sigma +1}}\right\}^{1/2}
\left\{\sum_{k=1}^{\infty}k^{\sigma +1}\|f_{s,k}\|^2\right\}^{1/2}\\
&\leqslant L\left\{\int\limits_{N}^{+\infty}\frac{dx}{x^{\sigma +1}}\right\}^{1/2}\|f_s\|_{\sigma}
=\frac{L}{\sqrt{\sigma}}\frac{1}{N^{\sigma/2}}\|f_s\|_{\sigma}\quad (N\in\mathbb N,\,\,\sigma >0).
\end{split}
\end{equation}

To estimate $D_2$, we  assume that
$0<\varepsilon_1<1$,\,
$0<\varepsilon_3 <1$,\,
$\varepsilon_1+\varepsilon_3<\sigma$,\,
$M>\sigma-\varepsilon_3$.
These conditions make it possible to apply lemma~\ref{lem4} and lemma~\ref{lem2}
for $n=j/2$ and $\alpha=(1+\sigma-\varepsilon_3)/2$.
We have
\begin{equation}\label{28}
\left\| \frac{D_2}{\min(x,1-x)}\right\|
=\left\|-\sum_{j=M+1}^{\infty}\frac{v_j(1-x)+v_j(x)}{\min(x,1-x)}
A^{-1}\big[(I+A)^{-1}A\big]^j f_0\right\|
\end{equation}
\begin{equation*}
\begin{split}
&\leqslant
\sum_{j=M+1}^{\infty}\frac{|v_j(1-x)|+|v_j(x)|}{\min(x,1-x)}
\left\|\frac{1}{2\pi i}\int\limits_{\Gamma}
\left(\frac{z}{1+z}\right)^jz^{-(1+\sigma)}(zI-A)^{-1}A^{\sigma}f_0\,dz \right\|\\
&\leqslant
\sum_{j=M+1}^{\infty}
 \frac{2C_1}{2\pi j^{(1-\varepsilon_1)/2}}
 \int\limits_{\Gamma}\left|\frac{z}{1+z}\right|^j |z|^{-(1+\sigma)}\frac{L}{1+|z|} \,|dz|\|A^{\sigma}f_0\|\\
 &\leqslant \frac{2C_1 L}{\pi}
 \sum_{j=M+1}^{\infty}\frac{1}{j^{(1-\varepsilon_1)/2}}
 \int\limits_0^{+\infty}\left(\frac{\rho^2}{1+\rho^2}\right)^{j/2}
 (\rho^2)^{-(1+\sigma-\varepsilon_3)/2}\frac{d\rho}{\rho^{\varepsilon_3}(1+\rho)}
 \|A^{\sigma}f_0\|\\
 &\leqslant \frac{2C_1 L}{\pi}
 \sum_{j=M+1}^{\infty} \frac{1}{j^{(1-\varepsilon_1)/2}}
 \sup_{t>0}\left[\left(\frac{t}{1+t}\right)^{j/2} t^{-(1+\sigma-\varepsilon_3)/2}\right]
 \int\limits_0^{+\infty}\frac{d\rho}{\rho^{\varepsilon_3}(1+\rho)}\|A^{\sigma}f_0\|\\
&\leqslant \frac{2C_1 L}{\sin(\pi\varepsilon_3)}
(1+\sigma-\varepsilon_3)^{(1+\sigma-\varepsilon_3)/2}
\sum\limits_{j=M+1}^{\infty}
\frac{1}{j^{(1+\sigma-\varepsilon_1-\varepsilon_3)/2}}\|A^{\sigma}f_0\|\\
&\leqslant
\frac{2C_1 L}{\sin(\pi\varepsilon_3)}
(1+\sigma-\varepsilon_3)^{(1+\sigma-\varepsilon_3)/2}
\int\limits_{M}^{+\infty}\frac{dx}{x^{(2+\sigma-\varepsilon_1-\varepsilon_3)/2}}
\|A^{\sigma}f_0\|\\
&=\frac{C_3}{M^{(\sigma-\varepsilon_1-\varepsilon_3)/2}}
\|A^{\sigma}f_0\|
\end{split}
\end{equation*}
with $C_3 =\frac{4C_1 L}{\sin(\pi\varepsilon_3)(\sigma-\varepsilon_1-\varepsilon_3)}
(1+\sigma-\varepsilon_3)^{(1+\sigma-\varepsilon_3)/2}$.

Next, since $D_3$ and $D_1$ are very much alike, we obtain
\begin{equation}\label{29}
\begin{split}
\left\| \frac{D_3}{\min(x,1-x)}\right\|&=
\left\|
\sum_{k=N+1}^{\infty}\frac{\sqrt{2}\big[\cos(2k\pi x)-1\big]}{\min(x,1-x)}
\big[(2k\pi )^2 I+A\big]^{-1}f_{c,k}
\right\|\\
&\leqslant \frac{L}{\sqrt{\sigma}}\frac{1}{N^{\sigma/2}}\|f_c\|_{\sigma}
\quad (N\in\mathbb N,\,\,\sigma >0).
\end{split}
\end{equation}

The analysis of $D_4$ is similar to that of $D_2$.
We assume once again that $0<\varepsilon_1<1$,\,
$0<\varepsilon_3 <1$,\,
$\varepsilon_1+\varepsilon_3<\sigma$,\,
$M>\sigma-\varepsilon_3$
and therefore apply lemma~\ref{lem4} and lemma~\ref{lem2}
for $n=j/2$ and $\alpha=(1+\sigma-\varepsilon_3)/2$.
Thus we have
\begin{equation}\label{30}
\left\| \frac{D_4}{\min(x,1-x)}\right\|
=\left\|
-\sum_{k=1}^{\infty}\sqrt{2}\big[(2k\pi)^2I+A\big]^{-1}
\sum_{j=M+1}^{\infty}
\frac{v_j(1-x)+v_j(x)}{\min(x,1-x)}
\big[(I+A)^{-1}A\big]^jf_{c,k}\right\|
\end{equation}
\begin{equation*}
\begin{split}
&\leqslant
\sum_{k=1}^{\infty}\sum_{j=M+1}^{\infty}
\frac{|v_j(1-x)|+ |v_j(x)|}{\min(x,1-x)}
\left\|\frac{\sqrt{2}}{2\pi i}\int\limits_{\Gamma}
\left(\frac{z}{1+z}\right)^j
\frac{z^{-\sigma}}{(2k\pi)^2+z}(zI-A)^{-1}A^{\sigma}f_{c,k}\,dz
\right\|
\end{split}
\end{equation*}
\begin{equation*}
\begin{split}
&\leqslant
\sum_{k=1}^{\infty}\sum_{j=M+1}^{\infty}
\frac{\sqrt{2}\,2C_1}{2\pi j^{(1-\varepsilon_1)/2}}
\int\limits_{\Gamma} \left|\frac{z}{1+z}\right|^j
\frac{|z|^{-\sigma}}{|(2k\pi)^2+z|}
\frac{L}{1+|z|}|dz|\,\|A^{\sigma} f_{c,k}\|\\
&\leqslant\frac{2\sqrt{2}C_1 L}{\pi}
\sum_{k=1}^{\infty}\sum_{j=M+1}^{\infty}
\frac{1}{j^{(1-\varepsilon_1)/2}}
\int\limits_0^{+\infty}
\left(\frac{\rho^2}{1+\rho^2}\right)^{j/2}
(\rho^2)^{-(1+\sigma-\varepsilon_3)/2}\frac{d\rho}{\rho^{\varepsilon_3}(1+\rho)}
\|A^{\sigma}f_{c,k}\|\\
&\leqslant\frac{2\sqrt{2}C_1 L}{\pi}
\sum_{k=1}^{\infty}\sum_{j=M+1}^{\infty}
\frac{1}{j^{(1-\varepsilon_1)/2}}
 \sup_{t>0}\left[\left(\frac{t}{1+t}\right)^{j/2} t^{-(1+\sigma-\varepsilon_3)/2}\right]
 \int\limits_0^{+\infty}\frac{d\rho}{\rho^{\varepsilon_3}(1+\rho)}\|A^{\sigma}f_{c,k}\|\\
  &\leqslant \frac{2\sqrt{2}C_1 L}{\sin(\pi\varepsilon_3)}
(1+\sigma-\varepsilon_3)^{(1+\sigma-\varepsilon_3)/2}
\sum\limits_{j=M+1}^{\infty}
\frac{1}{j^{(1+\sigma-\varepsilon_1-\varepsilon_3)/2}}
\sum_{k=1}^{\infty}\|A^{\sigma}f_{c,k}\|\\
&\leqslant
\frac{2\sqrt {2} C_1 L}{\sin(\pi\varepsilon_3)}
(1+\sigma-\varepsilon_3)^{(1+\sigma-\varepsilon_3)/2}
\int\limits_{M}^{+\infty}\frac{dx}{x^{(2+\sigma-\varepsilon_1-\varepsilon_3)/2}}
\|A^{\sigma}f_{c,k}\|
=\frac{\sqrt{2} C_3}{M^{(\sigma-\varepsilon_1-\varepsilon_3)/2}}
\|A^{\sigma}f_{c,k}\|
\end{split}
\end{equation*}
with the constant $C_3$ defined in~(\ref{28}).

And finally, we evaluate the summand $D_5$:
\begin{equation}\label{31}
\left\| \frac{D_5}{\min(x,1-x)}\right\|
=\left\|
-\sum_{k=N+1}^{\infty}\sqrt{2}\big[(2k\pi)^2I+A\big]^{-1}
\sum_{j=1}^{M}
\frac{v_j(1-x)+v_j(x)}{\min(x,1-x)}
\big[(I+A)^{-1}A\big]^jf_{c,k}\right\|
\end{equation}
\begin{equation*}
\begin{split}
&\leqslant
\sum_{k=N+1}^{\infty}\sum_{j=1}^{M}
\frac{|v_j(1-x)|+ |v_j(x)|}{\min(x,1-x)}
\left\|\frac{\sqrt{2}}{2\pi i}\int\limits_{\Gamma}
\left(\frac{z}{1+z}\right)^j
\frac{1}{(2k\pi)^2+z}(zI-A)^{-1}f_{c,k}\,dz
\right\|\\
&\leqslant
\sum_{k=N+1}^{\infty}\sum_{j=1}^{M}
\frac{2\sqrt{2}}{3\cdot2\pi}
\int\limits_{\Gamma} \left|\frac{z}{1+z}\right|^j
\frac{1}{|(2k\pi)^2+z|}
\frac{L}{1+|z|}|dz|\,\| f_{c,k}\|\\
&\leqslant \frac{L}{3\pi^2}
\sum_{k=N+1}^{\infty}\frac{1}{k}\|f_{c,k}\|  \sum_{j=1}^{M}
\int\limits_{0}^{+\infty}\frac{d\rho}{\sqrt{\rho}(1+\rho)}
=\frac{LM}{3\pi}\sum_{k=N+1}^{\infty}\frac{1}{k}\|f_{c,k}\| \\
&=\frac{LM}{3\pi}\sum_{k=N+1}^{\infty}\frac{1}{k^{(3+\sigma)/2}}
k^{(1+\sigma)/2}\|f_{c,k}\|
=\frac{LM}{3\pi}
\left\{\sum_{k=N+1}^{\infty}\frac{1}{k^{3+\sigma}}\right\}^{1/2}
\left\{\sum_{k=N+1}^{\infty}k^{1+\sigma}\|f_{c,k}\|^2\right\}^{1/2}\\
&\leqslant \frac{LM}{3\pi}
\left\{\int\limits_{N}^{+\infty} \frac{dx}{x^{3+\sigma}}\right\}^{1/2}\|f_c\|_{\sigma}
=\frac{L}{3\pi\sqrt{2+\sigma}}  \frac{M}{N^{1+\sigma/2}}\|f_c\|_{\sigma}
\quad (N\in\mathbb N,\,\,\sigma>-2).
\end{split}
\end{equation*}
Estimate~(\ref{26}) easily follows from inequalities (\ref{27})--(\ref{31})
for $M=N$.
\end{proof}

In the next theorem, we investigate error~(\ref{25})
under another set of assumptions concerning smoothness of $f(x)$.
\begin{theorem}\label{th4}
Let the following conditions hold true:
\begin{gather*}
\quad \nu=\frac{\cos\varphi}{L+1}, \quad f_0\in C(A,(1),\nu),\quad
f_{c,k}\in {C(A,(1),\nu)}\quad \forall k\in\mathbb N,\\
\|f_{c}\|_{A^{\infty}}\overset{\text{def}}{=}
\left\{\sum\limits_{k=1}^{\infty}\|f_{c,k}\|_{C(A,\,(1),\,\nu)}^{2}\right\}^{1/2}<\infty,\\
\|f_s\|_{\infty}\overset{\text{def}}{=}\sum\limits_{k=1}^{\infty}e^k\|f_{s,k}\|<\infty, \quad
\|f_c\|_{\infty}\overset{\text{def}}{=}\sum\limits_{k=1}^{\infty}e^k\|f_{c,k}\|<\infty.
\end{gather*}
Then the accuracy of the approximate solution~(\ref{23})
is characterized by the weighted estimate
\begin{equation}\label{32}
\begin{split}
\left\| \frac{u(x)-u_{N,N}(x)}{\min(x,1-x)}\right\|
&\leqslant \frac{C e^{-\sqrt{N+1}}}{{(N+1)}^{1/2-\varepsilon}}
\left( \|f_s\|_{\infty}+\|f_c\|_{\infty}+ \|f_0\|_{C(A,(1),\nu)} +\|f_c\|_{A^{\infty}}\right),\\
& x\in[0,\,1] \quad (N\in\mathbb N),
\end{split}
\end{equation}
where $\varepsilon >0$ is an arbitrary small  number and
$C>0$ is a constant independent of  $N$.
\end{theorem}
\begin{proof}
Except for the summand $D_2$, we use the integrations path~(\ref{11}).
For the sake of brevity, some obvious calculations
which are similar to those  in theorem~\ref{th3} and lemma~\ref{lem6} will be omitted.

Relying on~(\ref{27}), we evaluate $D_1$ as follows:
\begin{equation}\label{33}
\begin{split}
\left \| \frac{D_1}{\min(x,1-x)} \right\|
&= \left\|  \sum_{k=N+1}^{\infty}
\frac{\sqrt{2}\sin(2k\pi x)}{\min(x,1-x)}
\big[(2k\pi)^2I+A\big]^{-1}f_{s,k} \right\|\\
&\leqslant
L\sum_{k=N+1}^{\infty}\|f_{s,k}\|
=L\sum_{k=N+1}^{\infty}e^{-k}e^{k}\|f_{s,k}\|\\
&\leqslant
L e^{-(N+1)}\sum_{k=N+1}^{\infty}e^k \|f_{s,k}\|
\leqslant
\frac {L e^{-\sqrt{N+1}}}{\sqrt{N+1}}\|f_s\|_{\infty}.
\end{split}
\end{equation}

To estimate $D_2$, we take the integration path $\widetilde{\Gamma}$
consisting of two rays and a circle arc:
\begin{equation*}
\begin{split}
\widetilde{\Gamma}=\widetilde{\Gamma}_{-}\cup\widetilde{\Gamma}_{+}\cup\Gamma_{\gamma},\quad
\widetilde{\Gamma}_{\pm}&=\big\{ z\in \mathbb C\mid z=\rho e^{\pm i\varphi },\,\rho\in[\gamma,\,+\infty)\big\},\\
\Gamma_{\gamma}&= \{ z\in \mathbb C\mid z=\gamma e^{i\theta },\,\theta\in[-\varphi,\,+\varphi]\big\},
\end{split}
\end{equation*}
with $dz=d(\rho e^{\pm i \varphi}) = e^{\pm i\varphi}d\rho$ \,for $\widetilde{\Gamma}_{\pm}$\,
and \,$dz = d({\gamma e^{i\theta}})=i\gamma e^{i\theta}d\theta$ \,for $\Gamma_{\gamma}$.
Then
\begin{equation*}
\left\| \frac{D_2}{\min(x,1-x)}\right\|
=\left\|-\sum_{j=M+1}^{\infty}\frac{v_j(1-x)+v_j(x)}{\min(x,1-x)}
A^{-1}\big[(I+A)^{-1}A\big]^j f_0\right\|
\end{equation*}
\begin{equation*}
\begin{split}
&\leqslant
\sum_{j=M+1}^{\infty}\frac{|v_j(1-x)|+|v_j(x)|}{\min(x,1-x)}
\Bigg\|\frac{1}{2\pi i}\int\limits_{\widetilde{\Gamma}}
\frac 1 z \left(\frac{z}{1+z}\right)^j
\left(1+\frac z j\right)^{-j}(zI-A)^{-1}\left(I+\frac A j\right)^j f_0\,dz \Bigg\|\\
&\leqslant
\sum_{j=M+1}^{\infty}
 \frac{2C_1}{2\pi j^{(1-\varepsilon_1)/2}}
 \int\limits_{\widetilde{\Gamma}}
 \frac {1}{|z|}\left|\frac{z}{1+z}\right|^j \left | 1+\frac{z}{j} \right|^{-j}
  \frac{L}{1+|z|} \,|dz|\,\bigg\|\left(I+\frac A j \right)^j f_0\bigg\|\\
 &\leqslant \frac{2C_1 L}{\pi}
 \sum_{j=M+1}^{\infty}\frac{1}{j^{(1-\varepsilon_1)/2}}
 \left\{
 \int\limits_{\gamma}^{+\infty}
 \left [
 \frac{\rho\cos\varphi}
 {\sqrt{1+2\rho\cos\varphi+\rho^2}\sqrt{1+2\frac{\rho}{j}\cos\varphi +\frac{\rho^2}{j^2}}}
 \right ]^j
 \frac{d\rho}{\rho(1+\rho)}
 \right.\\
 &\quad +\int\limits_0^{\varphi}
\left.
 \left [
 \frac{\gamma\cos\theta}
 {\sqrt{1+2\gamma\cos\theta+\gamma^2}\sqrt{1+2\frac{\gamma}{j}\cos\theta +\frac{\gamma^2}{j^2}}}
 \right ]^j
 \frac{d\theta}{1+\gamma}
 \right\}   \cos^{-j}\varphi\,
 \bigg\|\left(I+\frac A j \right)^j A^{-j}A^{j} f_0\bigg\|\\
 &\leqslant \frac{2C_1 L}{\pi}
 \sum_{j=M+1}^{\infty}\frac{1}{j^{(1-\varepsilon_1)/2}}
 \left\{
 \int\limits_{\gamma}^{+\infty}
 \left [
 \frac{\rho\cos\varphi}
 {(1+\rho\cos\varphi)\big(1+\frac{\rho\cos\varphi}{j}\big)}
 \right ]^j
 \frac{d\rho}{\rho(1+\rho)}
 \right.\\
 &\quad +\int\limits_0^{\varphi}
\left. \left [
\frac{\gamma\cos\theta}
 {(1+\gamma\cos\theta)\big(1+\frac{\gamma\cos\theta}{j}\big)}
\right ]^j \frac{d\theta}{1+\gamma} \right\} \cos^{-j}\varphi\,
 \bigg\|\left(I+\frac A j \right)^j A^{-j}\bigg\| \|  A^{j} f_0\|\\
 &\leqslant \frac{2C_1 L}{\pi}
 \sum_{j=M+1}^{\infty}\frac{1}{j^{(1-\varepsilon_1)/2}}
\max_{t\geqslant 0}\left[\frac{t}{(1+t)\big(1+\frac t j\big)}\right]^j
 \left\{
  \int\limits_{\gamma}^{+\infty}
  \frac{d\rho}{\rho(1+\rho)}
  +\int\limits_0^{\varphi}
\frac{d\theta}{1+\gamma} \right\} \\
&\qquad\qquad\qquad \times \cos^{-j}\varphi\,
 (L+1)^j\|  A^{j} f_0\|.
\end{split}
\end{equation*}
Applying here lemma~\ref{lem3}, we obtain
\begin{equation}\label{34}
\begin{split}
\left\| \frac{D_2}{\min(x,1-x)}\right\|
&\leqslant \frac{2C_1 L}{\pi}
\left(\ln{\frac{1+\gamma}{\gamma}}+\frac{\varphi}{1+\gamma}\right)
 \sum_{j=M+1}^{\infty}\frac{ee^{-2\sqrt{j}}}{j^{(1-\varepsilon_1)/2}}
 \|f_0\|_{C(A,\,(1),\,\nu)}\\
&\leqslant
\frac{C_4e^{-\sqrt{M+1}}}{(M+1)^{(1-\varepsilon_1)/2}}
\|f_0\|_{C(A,\,(1),\,\nu)}
\end{split}
\end{equation}
with $C_4 =\frac{2C_1 L\widetilde{S}e}{\pi}
\left(\ln{\frac{1+\gamma}{\gamma}}+\frac{\varphi}{1+\gamma}\right)$,
where $\widetilde{S}$ is the sum of the convergent number series
$\widetilde{S}=\sum\limits_{j=1}^{\infty}
e^{-\sqrt{j}}  =1{.}670406818\ldots$.

The summand $D_3$ is estimated in the same way as  $D_1$:
\begin{equation}\label{35}
\begin{split}
\left\| \frac{D_3}{\min(x,1-x)}\right\|&=
\left\|
\sum_{k=N+1}^{\infty}\frac{\big[\sqrt{2}\cos(2k\pi x)-1\big]}{\min(x,1-x)}
\big[(2k\pi )^2 I+A\big]^{-1}f_{c,k}
\right\|\\
&\leqslant L\frac {e^{-\sqrt{N+1}}}{\sqrt{N+1}}\|f_c\|_{\infty}.
\end{split}
\end{equation}

Next, we have
\begin{equation}\label{36}
\left\| \frac{D_4}{\min(x,1-x)}\right\|
=\left\|
-\sum_{k=1}^{\infty}\sqrt{2}\big[(2k\pi)^2I+A\big]^{-1}
\sum_{j=M+1}^{\infty}
\frac{v_j(1-x)+v_j(x)}{\min(x,1-x)}
\big[(I+A)^{-1}A\big]^jf_{c,k}\right\|
\end{equation}
\begin{equation*}
\begin{split}
&\leqslant
\sum_{k=1}^{\infty}\sum_{j=M+1}^{\infty}
\frac{|v_j(1-x)|+ |v_j(x)|}{\min(x,1-x)}\\
&\qquad\times
\Bigg\|\frac{\sqrt{2}}{2\pi i}\int\limits_{\Gamma}
\frac 1 {(2k\pi)^2+z} \left(\frac{z}{1+z}\right)^j
\left(1+\frac z j\right)^{-j}(zI-A)^{-1}\left(I+\frac A j\right)^j f_{c,k}\,dz \Bigg\|\\
&\leqslant
\sum_{k=1}^{\infty}\sum_{j=M+1}^{\infty}
\frac{\sqrt{2}\,2C_1}{2\pi j^{(1-\varepsilon_1)/2}}
\int\limits_{\Gamma}
\frac{1}{|(2k\pi)^2+z|}
\left|\frac{z}{1+z}\right|^j
\left| 1+\frac{z}{j}\right|^{-j}
\frac{L}{1+|z|}|dz|\,
\bigg\|\left(I+\frac A j \right)^j f_{c,k}\bigg\|\\
&\leqslant\frac{2\sqrt{2}C_1 L}{\pi}
\sum_{k=1}^{\infty}\sum_{j=M+1}^{\infty}
\frac{1}{j^{(1-\varepsilon_1)/2}}
\int\limits_0^{+\infty}
\frac{1}{2k\pi\sqrt{2\rho}}
\left [
\frac{\rho\cos\varphi}
{(1+\rho\cos\varphi)\big(1+\frac{\rho\cos\varphi}{j}\big)}
  \right ]^j  \frac{d\rho}{1+\rho}\\
&\qquad \times
\cos^{-j}\varphi\,\bigg\|\left(I+\frac A j \right)^j
A^{-j}\bigg\| \|A^{j}f_{c,k}\|\\
&\leqslant\frac{C_1 L}{\pi^2}
\sum_{k=1}^{\infty}\sum_{j=M+1}^{\infty}
\frac{1}{kj^{(1-\varepsilon_1)/2}}
 \max_{t\geq 0}
 \left[\frac{t}{(1+t)\big(1+\frac{t}{j}\big)}\right]^j
 \int\limits_0^{+\infty} \frac{d\rho}{\sqrt{\rho}(1+\rho)}
 \cos^{-j}\varphi\,(L+1)^j\|A^j f_{c,k}\|  \\
&\leqslant \frac{C_1 L}{\pi}
\sum\limits_{j=M+1}^{\infty}
\frac{e e^{-2\sqrt{j}}}{j^{(1-\varepsilon_1)/2}}
\sum_{k=1}^{\infty}\frac{1}{k} \|f_{c,k}\|_{C(A,\,(1),\,\nu)}\\
&\leqslant
\frac{ C_1 L\widetilde{S}e}{\pi}
\frac{e^{-\sqrt{M+1}}}{(M+1)^{(1-\varepsilon_1)/2}}
\left\{\sum\limits_{k=1}^{\infty}\frac{1}{k^2}\right\}^{1/2}
\left\{\sum_{k=1}^{\infty} \|f_{c,k}\|^2_{C(A,\,(1),\,\nu)}\right\}^{1/2}\\
&\leqslant \frac{C_5 e^{-\sqrt{M+1}}}{(M+1)^{(1-\varepsilon)/2}}
\|f_c\|_{A^{\infty}}
\end{split}
\end{equation*}
with the constant $C_5=\frac{C_1 L\widetilde{S}e}{\sqrt{6}}$
and $\widetilde{S}$  described in~(\ref{34}).

At last, it remains to consider $D_5$. Taking into account~(\ref{31}), we get
\begin{equation}\label{37}
\begin{split}
\left\| \frac{D_5}{\min(x,1-x)}\right\|
&=\left\|
-\sum_{k=N+1}^{\infty}\sqrt{2}\big[(2k\pi)^2I+A\big]^{-1}
\sum_{j=1}^{M}
\frac{v_j(1-x)+v_j(x)}{\min(x,1-x)}
\big[(I+A)^{-1}A\big]^jf_{c,k}\right\|\\
&\leqslant \frac{LM}{3\pi}\sum_{k=N+1}^{\infty}
\frac{1}{k}\|f_{c,k}\|
= \frac{LM}{3\pi}\sum_{k=N+1}^{\infty}\frac {e^{-k}}{k}e^{k}\|f_{c,k}\|\\
&\leqslant \frac{LMe^{-(N+1)}}{3\pi(N+1)} \sum_{k=1}^{\infty}\|f_{s,k}\|
\leqslant \frac{C_6 e^{-\sqrt{N+1}}}{\sqrt{N+1}} \sum_{k=1}^{\infty}\|f_{s}\|_{\infty}
\end{split}
\end{equation}
with the constant $C_6=\frac{L}{3\pi}$.

Now estimates~(\ref{33})--(\ref{37}) with $M=N$ easily lead to the assertion of the theorem.
\end{proof}

\section{Conclusion}
To summarize, we make some general comments on the results proven above.

%1.
In theorems~\ref{th1} and \ref{th2} for the homogeneous equation
and in theorems~\ref{th3} and \ref{th4} for the inhomogeneous one,
the boundary effect is evaluated through the weight function $\min{(x,\,1-x)}$
%related to the interval $[0,1]$ for the independent variable.
which characterizes the distance from the boundary points of the interval $[0,1]$.

%2.
Both estimate~(\ref{15}) and estimate~(\ref{26}) are power-dependent
on the parameter $\sigma$.
%Estimate (\ref{15})
The first one
indicates that if  $\sigma$ increases
(meaning that differential properties of $u_1$ improve),
the convergence rate of the approximate solution $u_N(x)$  automatically gets higher.
Therefore, method~(\ref{14})
%for  BVP~(\ref{1})
is a method without saturation of accuracy.
The same is true of the second estimate.
%estimate~(\ref{26}).
Namely, when $\sigma$ increases
(implying that the Fourier coefficients decay more quickly, i.e. $f(x)$ has  better smoothness),
the convergence rate of the approximate solution $u_{N,N}(x)$ accordingly goes up.
So, method~(\ref{23})
%for BVP~(\ref{18})
does not have saturation of accuracy either.

%3.
Estimate~(\ref{16}) shows that method~(\ref{14}) %for BVP~(\ref{1})
has the exponential rate of convergence
provided that $u_1$ is a vector of the exponential type in the sense of~\cite{Radyno1985}.
In a similar way, estimate~(\ref{32}) indicates that method~(\ref{23}) %for BVP~(\ref{18})
has the exponential rate of convergence provided
vectors $f_0$, $f_{s,k}$, $f_{c,k}$, $k=1,2,\dots$,
in the Fourier representation of $f(x)$
have  appropriate exponential rates of decay.
Both of these methods %(\ref{14}) and (\ref{23})
can  be called
\textit{super-exponentially convergent}.

\bibliography{Makarov_Mayko}

\end{document}